\documentclass[11pt,a4paper]{amsart}
\usepackage[margin=3cm]{geometry}
\usepackage{color}               
\usepackage{faktor}
\usepackage{tikz}
\usepackage{tikz-cd}
\usepackage{spverbatim}

\usepackage[pdftex,hyperfootnotes=false,colorlinks=false,hypertexnames=false]{hyperref}

\usepackage{amsmath,amssymb,anyfontsize,enumerate,stmaryrd,mathtools, thmtools,tikz,txfonts,nccmath}
\usepackage[oldstyle]{libertine}%
\usepackage[T1]{fontenc}
\usepackage[utf8]{inputenc}
\usepackage{csquotes}
\makeatletter
\renewcommand*\libertine@figurestyle{LF}
\makeatother
\usepackage[libertine,libaltvw,liby]{newtxmath}
\makeatletter
\renewcommand*\libertine@figurestyle{OsF}
\makeatother
\usepackage[noerroretextools,backend=biber,style=alphabetic,maxalphanames=5,giveninits,sorting=nyt,%
maxbibnames=99]{biblatex}
\usepackage{tikz-cd}
\usepackage[noabbrev]{cleveref}
\usepackage{autonum}

\crefname{lemma}{lemma}{lemmata}
\Crefname{lemma}{Lemma}{Lemmata}
\crefname{subsection}{subsection}{subsections}
\Crefname{subsection}{Subsection}{Subsections}

\newtheorem{theorem}{Theorem}[section]
\newtheorem{lemma}[theorem]{Lemma}
\newtheorem{proposition}[theorem]{Proposition}
\newtheorem{corollary}[theorem]{Corollary}

\theoremstyle{definition}
\newtheorem{definition}[theorem]{Definition}
\newtheorem{remark}[theorem]{Remark}

\newtheorem{example}[theorem]{Example}
\newtheorem{construction}[theorem]{Construction}

\def\CP1{\mathbb{C}\mathrm{P}^1}

\newcommand\restr[2]{{
  \left.\kern-\nulldelimiterspace 
  #1 
  \vphantom{\big|} 
  \right|_{#2} 
  }}

\newcommand\restrst[2]{{
  \left.\kern-\nulldelimiterspace 
  #1 
  \vphantom{\big|} 
  \right|^*_{#2} 
  }}

\begingroup\expandafter\expandafter\expandafter\endgroup
\expandafter\ifx\csname pdfsuppresswarningpagegroup\endcsname\relax
\else
\fi
{}

\title{One-parameter families of multiview varieties via quotient lattices}
\author[M.~A.~Hahn]{Marvin Anas Hahn}
\address{M.~A.~Hahn: Max--Planck Institute for Mathematics in the Sciences, Inselstraße 22, 04103 Leipzig}
\email{mhahn@mis.mpg.de}


\begin{document}
\begin{abstract}
We develop a novel theory of one-parameter families of multi-view varieties. These families are induced by quotient lattices over discrete valuation rings and generalise the notion of \textit{Mustafin varieties}. We study the geometry and the combinatorics of the limit of these families.
\end{abstract}
\maketitle
\tableofcontents

\section{Introduction}
In recent years, the application of algebraic geometry in the field of computer vision has attracted a lot of interest, particularly in the context of so-called \textit{multi-view geometry} \cite{,aholt2013hilbert,joswig2016rigid,duff2019plmp,lieblich2020two,agarwal2019ideals,ponce2017congruences}. The field of multi-view geometry is concerned with the situation of several cameras taking pictures of the same object simultaneously. For so-called \textit{pinhole cameras} taking a picture is modelled as a projection, i.e. each camera is represented by a $3\times 4$ matrix $A_i$ of rank $3$ over a field $K$. For a tuple $\underline{A}=(A_1,\dots,A_n)$ of such matrices -- where each matrix represents a camera -- one therefore considers the rational map
\begin{equation}
\begin{tikzcd}
f_{\underline{A}}\colon\mathbb{P}^3 \arrow[dashed,"{(A_1,\dots,A_n)}"]{rr} & & \left(\mathbb{P}^2\right)^n
\end{tikzcd}
\end{equation}
which we call the \textit{vision map}. Furthermore, we denote $M_{\underline{A}}=\overline{\mathrm{Im}(f_{\underline{A}})}$ and call $M_{\underline{A}}$ the \textit{multi-view variety} associated to $\underline{A}$. Multi-view varieties are central objects in reconstruction problems in computer vision \cite{hartley2003multiple}.\\
It is natural to generalise this to the following situation: Let $d,n,l_1,\dots,l_n$ be positive integers with $l_i\le d$ and $A_i$ a $l_i\times d$ matrix. The tuple $\underline{A}=(A_1,\dots,A_n)$ induces a rational map
\begin{equation}
\begin{tikzcd}
f_{\underline{A}}\colon\mathbb{P}^{d-1} \arrow[dashed,"{(A_1,\dots,A_n)}"]{rr} & & \prod_{i=1}^n\mathbb{P}^{l_i-1}
\end{tikzcd}
\end{equation}
which we call the \textit{generalised vision map}. As above, we denote the closure of the image of $f_{\underline{A}}$ by $M_{\underline{A}}$ and call this the \textit{generalised multi-view variety} associated to $\underline{A}$. Generalised multi-view varieties were first studied in \cite{conca2017cartwright} in the context of so-called Cartwright--Sturmfels ideals. We note that for $d=4,l_i=3$, we obtain the multi-view varieties associated to pinhole cameras defined above, while for $d=4,l_i=2$ and $n$ even generalised multi-view varieties specialise to the situation of so-called \textit{two-slit cameras} considered in \cite{ponce2017congruences}.\\
In recent years, the study of families of multi-view varieties has gained a lot of interest, in particular due to surprising connections to Hilbert schemes \cite{aholt2013hilbert,lieblich2020two}. In this paper, we develop a novel theory of one-parameter families of multi-view varieties induced by a finite set of quotient-lattices over a discrete valuation ring that we call \textit{Mustafin degenerations of generalised multi-view varieties}. Our construction of these families is motivated by the notion of so-called \textit{Mustafin varieties} in arithmetic geometry. These objects were introduced by Mustafin \cite{mustafin1978nonarchimedean} in 1970s and are degenerations of projective spaces induced by a finite set of lattices. Since their introduction, Mustafin varieties have given rise to applications in the theory of Shimura varieties, Chow quotients of Grassmannians, tropical geometry, limit linear series theory and arithmetic geometry \cite{fa,keel2006geometry,cartwright2011mustafin,hahn2017mustafin,he2019linked,HWplane,hahn2020mustafin}. Our work may also be viewed as a generalisation of Mustafin varieties which is of independent interest.\\
In \cref{sec-con}, we construct Mustafin degenerations as schemes over a discrete valuation ring. We are particularly interested in their special fibres which may be viewed as the "limit" of these families. We show in \cref{thm-main1} that Mustafin degenerations are integral, flat and projective schemes. Furthermore, we show that their special fibres are connected, generically reduced and we give an upper bound on the number of irreducible components of this limit. Further, we show that for $d=4$, $l_i=3$ and $A_i$ generic, the Mustafin degenerations of these multi-view varieties are Cohen-Macaulay with reduced generic fibre. In \cref{prop-class}, we study irreducible component sof the special fibre of Mustafin degenerations and obtain a partial classification. Finally, in \cref{ex-genericin} we show that the generic initial ideal of multiview varieties studied in \cite{aholt2013hilbert} is attained as the limit of a Mustafin degeneration.


\subsection*{Acknowledgements} We would like to thank Bernd Sturmfels and Annette Werner for many helpful conversations regarding this work. Parts of this work were completed with support of the LOEWE research unit Uniformized Structures in Arithmetic and Geometry.

\section{Preliminaries}
In this section, we review the necessary prerequisites for this paper.
\subsection{Images of rational maps}
A main observation for this paper, is that the vision map $f_{\underline{A}}$ factorises as follows
\begin{equation}
\label{equ-diagram}
\begin{tikzcd}
\mathbb{P}^{d-1}\arrow[dashed,"g_{\underline{A}}"]{rr} \arrow[dashed,"f_{\underline{A}}",swap]{drr}& & \prod_{i=1}^n\mathbb{P}\left(\faktor{K^d}{\mathrm{Ker}(A_i)}\right)\arrow["{(A_1,\dots,A_n)}"]{d}\\
& &  \prod_{i=1}^n\mathbb{P}^{l_i-1}
\end{tikzcd}
\end{equation}
where the horizontal map $g_{\underline{A}}$ is the natural quotient map in each factor. We observe that the vertical morphism $(A_1,\dots,A_n)$ induces a linear isomorphism between $\overline{\mathrm{Im}(g_{\underline{A}})}$ and $M_{\underline{A}}$.\\
A conceptual study of varieties of the form $\overline{\mathrm{Im}(g_{\underline{A}})}$ was undertaken in \cite{Lirational}. In particular, the results in \cite{Lirational} may be rephrased to an explicit description of the Chow class of $M_{\underline{A}}\subset\prod\mathbb{P}^{l_i-1}$. In order to give this description, recall that the Chow ring of $\prod\mathbb{P}^{l_i-1}$ is given by
\begin{equation}
A\left(\prod_{i=1}^n\mathbb{P}^{l_i-1}\right)=\faktor{\mathbb{Z}[H_1,\dots,H_n]}{\langle H_1^{l_1},\dots,H_n^{l_n}\rangle}
\end{equation}
where $H_i$ is represents the hyperplane class in the $i-$th factor. Now, let $I\subset\{1,\dots,n\}$ and denote $d_I=\bigcap_{i\in I}\mathrm{ker}(A_i)$. We define the set
\begin{equation}
M(p)\coloneqq\{\underline{m}\in\mathbb{Z}^n_{\ge0}\mid d-\sum_{i\in I}m_i>d_I,\,\sum_{i=1}^nm_i=p\}.
\end{equation}
Furthermore, we define
\begin{equation}
p_0=\mathrm{max}_{M(p)\neq\emptyset}(p).
\end{equation}

The following is a reformulation of \cite[Theorem 1.1]{Lirational}.

\begin{theorem}[{\cite[Theorem 1.1]{Lirational}}]
\label{thm-bing}
In the notation above, let the base field $K$ be infinite. Then, we have $\mathrm{dim}(M_{\underline{A}})=p_0$ and its Chow class in $\prod\mathbb{P}^{l_i-1}$ is given by
\begin{equation}
\sum_{\underline{m}\in M(p_0)}\prod H_i^{l_i-1-m_i}.
\end{equation} 
\end{theorem}
We note that while \cite[Theorem 1.1]{Lirational} is originally stated for algebraically closed fields, the proof is valid for any infinite field.\\

\subsection{Bruhat--Tits building of $\mathrm{GL}(V)$}
We now fix $K$ a discretely valued field with ring of integers $\mathcal{O}_K$, uniformiser $\pi$, maximal ideal $\mathfrak{m}_K$ and residue field $k\cong\faktor{\mathcal{O}_K}{\mathfrak{m}_K}$. Readers that are not familiar with valued field may think of $K=\mathbb{Q}((t))=\{\sum_{i\ge n}a_it^i\mid n\in\mathbb{Z}\}$, which yields $\mathcal{O}_K=\mathbb{Q}[[t]]=\{\sum_{i\ge n}a_it^i\mid n\in\mathbb{Z}_{\ge0}\}$, $\mathfrak{m}_K=\{\sum_{i\ge n}a_it^i\mid n\in\mathbb{Z}_{>0}\}$ and $k\cong\mathbb{Q}$.

Bruhat--Tits buildings are geometric realisations of simplicial complexes associated to reductive algebraic groups. In this paper, we will focus only on the Bruhat--Tits building $\mathfrak{B}_d$ associated to $\mathrm{GL}(V)$ for a $K-$vector space $V$ of dimension $d$. More precisely, we will only consider the integral points of these buildings. We begin with the following definition.

\begin{definition}
We call a free rank $d$ $\mathcal{O}_K-$submodule $L$ of $V$ a \textit{lattice}. Furthermore, let $M\subset L$ a free submodule. If $\faktor{L}{M}$ is a free module, we call $\faktor{L}{M}$ a \textit{quotient lattice}.\\
Furthermore, we define the \textit{homothety class} of a lattice $L$ by
\begin{equation}
    [L]=\{\alpha \cdot L\mid \alpha\in K^\ast\}.
\end{equation}

For $d=\mathrm{dim}(V)$, we denote the set of homothety classes of lattices in $V$ by $\mathfrak{B}_d^0$. We call two homothety classes $[L],[M]\in\mathfrak{B}_d^0$ adjacent if there exist representatives $L\in[L],M\in[M]$, such that $\pi M\subset L\subset M$.\\
Similarely, we define the homothety class of a quotient lattice $\faktor{L}{M}$ by
\begin{equation}
    \left[\faktor{L}{M}\right]=\left\{\faktor{\alpha\cdot L}{\alpha\cdot M}\right\}\mid \alpha\in K^\ast\}.
\end{equation}
\end{definition}

The set $\mathfrak{B}_d^0$ may be viewed as the integral points of the Bruhat--Tits building $\mathfrak{B}_d$ associated to $\mathrm{GL}(V)$.  The adjacency relation on $\mathfrak{B}_d^0$ induces a simplicial complex structure on $\mathfrak{B}_d^0$, i.e. a subset $\Gamma\subset\{[L_1],\dots,[L_n]\}\subset\mathfrak{B}_d^0$ forms a simplex, if its elements are pairwise adjacent. An important notion in the theory of Bruhat--Tits buildings is that of convexity.

\begin{definition}
We call a subset $\Gamma\subset\mathfrak{B}_d^0$ \textit{convex} if for all $[L],[L']\in\Gamma$, we have $[\pi^aL\cap\pi^bL']\in\Gamma$ for all $a,b\in\mathbb{Z}$.\\
For a subset $\Gamma\subset\mathfrak{B}_d^0$, we define its \textit{convex hull} $\mathrm{conv}(\Gamma)$ to be the smallest convex set in $\mathfrak{B}_d^0$ containing $\Gamma$.
\end{definition}

%

\subsection{Mustafin varieties}
\label{subsec-must}
We now discuss the basic notions revolving around \textit{Mustafin varieties} (see also \cite{cartwright2011mustafin,hahn2017mustafin}). Let $V$ be a vector space of dimension $d$ over $K$ and $L$ a lattice. We further define
\begin{equation}
\mathbb{P}(V)=\mathrm{Proj}\mathrm{Sym} V^\ast\quad\textrm{and}\quad \mathbb{P}(L')=\mathrm{Proj}\mathrm{Sym}\left(\mathrm{Hom}_{\mathcal{O}_K}(L,\mathcal{O}_K)\right),
\end{equation}

i.e. we define projective spaces to parametrise lines (and not hyperplanes). We are now ready to define Mustafin varieties.

\begin{definition}
\label{def:musta}
Let $\Gamma=\{[L_1],\dots,[L_n]\}$ be a set of homothety classes of rank $d$ lattice classes in $V$. Then $\mathbb{P}(L_1),\dots,\allowbreak\mathbb{P}(L_n)$ are projective spaces over $R$ whose generic fibres are canonically isomorphic to $\mathbb{P}(V)\simeq\mathbb{P}^{d-1}_{K}$. The open immersions
\begin{equation}
\mathbb{P}(V)\hookrightarrow\mathbb{P}(L_i)
\end{equation}
give rise to a map
\begin{equation}
\label{equ:mapmust}
f_{\Gamma}:\mathbb{P}(V)\longrightarrow\mathbb{P}(L_1)\times_R\dots\times_R\mathbb{P}(L_n).
\end{equation}
We denote the closure of the image endowed with the reduced scheme structure by $\mathcal{M}(\Gamma)$. We call $\mathcal{M}(\Gamma)$ the \textit{associated Mustafin variety}. Its special fibre $\mathcal{M}(\Gamma)_k$ is a reduced scheme over $k$ by \cite[Theorem 2.3]{cartwright2011mustafin}.\end{definition}

There is a natural way to choose coordinates on Mustafin varieties. For this, we fix a reference lattices $L=\mathcal{O}_Ke_1+\dots+\mathcal{O}_Ke_d$, where $e_1,\dots,e_d$ is the standard basis of $V$. For $L_1,\dots,L_n$ as in \cref{def:musta}, we can find $g_1,\dots,g_n\in\mathrm{PGL}(V)$ such that $g_iL=L_i$. We consider the commutative diagram
\begin{equation}
\begin{tikzcd}
\mathbb{P}(V)\arrow{rr}{(g_1^{-1},\dots,g_n^{-1})\circ\Delta} \arrow{d} &  & \mathbb{P}(V)^n\arrow{d}\\
\prod_{R}\mathbb{P}(L_i) \arrow{rr}{(g_1^{-1},\dots,g_n^{-1})} &  & \mathbb{P}(L)^n.
\end{tikzcd}
\end{equation}
Let $x_1,\dots,x_d$ be the coordinates on $\mathbb{P}(L)$ and consider the projections
\begin{equation}
P_j:\mathbb{P}(L)^n\to\mathbb{P}(L)
\end{equation}
to the $j-$th factor. Then, we denote $x_{ij}=P_j^*x_i$ and observe that the Mustafin variety $\mathcal{M}(\Gamma)$ is isomorphic to the subscheme of $\mathbb{P}(L)^n$ cut out by
\begin{equation}
\label{equ-idmust}
I_2\begin{pmatrix}
g_1\begin{pmatrix}
x_{11}\\
\vdots\\
x_{d1}
\end{pmatrix} &
\cdots
& g_n\begin{pmatrix}
x_{1n}\\
\vdots\\
x_{dn}
\end{pmatrix}
\end{pmatrix}\cap \mathcal{O}_K[(x_{ij})].
\end{equation}
By
\begin{equation}
p_j=\restr{P_j}{\mathcal{M}(\Gamma)}:\mathcal{M}(\Gamma)\hookrightarrow\mathbb{P}(L)^n\to\mathbb{P}(L)
\end{equation}
we denote the projection to the $j-$th component. We write $x_{ij}$ also for the induced rational function on $\mathcal{M}(\Gamma)$.

\begin{theorem}[{\cite{mustafin1978nonarchimedean,fa,cartwright2011mustafin,hahn2017mustafin}}]
\label{thm-must}
\begin{enumerate}
 \item For a finite set of lattice classes $\Gamma\subset\mathfrak{B}_d^0$, the Mustafin variety $\mathcal{M}(\Gamma)$ is an integral, normal, Cohen-Macaulay scheme which is flat and projective over $R$. Its generic fiber is isomorphic to $\mathbb{P}^{d-1}_{\mathbb{K}}$ and its special fiber is reduced, Cohen-Macaulay and connected. All irreducible components are rational varieties and their number is at most ${n+d-2 \choose d-1}$, where $n=|\Gamma|$.
 \item If $\Gamma$ is a convex set in $\mathfrak{B}_d^0$, then the Mustafin variety is regular and its special fiber consists of $n$ smooth irreducible components that intersect transversely. In this case the reduction complex of $\mathcal{M}(\Gamma)$ is induced by the simplicial subcomplex of $\mathfrak{B}_d$ induced by $\Gamma$.
 \item An irreducible component mapping birationally to the special fiber of $\mathbb{P}(L_i)$ is called a \textit{primary component}. The other components are called \textit{secondary components}. For each $i=1,\dots,n$ there exists such a primary component. A projective variety $X$ arises as a primary component for some subset $\Gamma\subset\mathfrak{B}_d$ if and only if $X$ is the blow-up of $\mathbb{P}_k^{d-1}$ at a collection of $n-1$ linear subspaces.
 \item Let $C$ be a secondary component of $\mathcal{M}(\Gamma)_k$. There exists an element $v$ in $\mathrm{conv}(\Gamma)$, such that 
 \begin{equation}
 \mathcal{M}(\Gamma\cup\{v\})_k\longrightarrow\mathcal{M}(\Gamma)_k
 \end{equation}
  restricts to a birational morphism $\tilde{C}\rightarrow C,$  where $\tilde{C}$ is the primary component of $\mathcal{M}(\Gamma\cup\{v\})_k$ corresponding to $v$.
  \item The irreducible components of $\mathcal{M}(\Gamma)_k$ are of the form $\overline{\mathrm{Im}(g)}$ where
  \begin{equation}
  \begin{tikzcd}
      g\colon \mathbb{P}(k^d) \arrow[dashed]{r}& \prod_{i=1}^n\mathbb{P}\left(\faktor{k^d}{V_i}\right)
  \end{tikzcd}
  \end{equation}
    for $V_i\subset k^d$.
  \end{enumerate}
\end{theorem}










\section{Constructing one-parameter families of generalised multi-view varieties}
\label{sec-con}

In this section, we introduce our main object of study, i.e. \textit{Mustafin degenerations of generalised multiview varieties}. To begin with, let $L'=\faktor{L}{M}$ be a quotient lattice and define

\begin{equation}
    \mathrm{Proj}\mathrm{Sym}(\mathrm{Hom}(L',\mathcal{O}_K)).
\end{equation}

The following definition introduces Mustafin degenerations of generalised multiview varieties.

\begin{definition}
Let $\Gamma=\left\{\left[\faktor{L_1}{M_1}\right],\dots,\left[\faktor{L_n}{M_n}\right]\right\}$ be a set of quotient lattices. Then, for each $i$, there is a canonical rational map $f_i:\mathbb{P}(V)\dashrightarrow\mathbb{P}\left(\faktor{L_i}{M_i}\right)$ onto the generic fibre of $\mathbb{P}\left(\faktor{L_i}{M_i}\right)$. Thus, we obtain a rational map
\begin{equation}
\begin{tikzcd}
\underline{f}:\mathbb{P}(V) \arrow[dashed,"{(f_1,\dots,f_n)}"]{rr} & & \prod_{i=1}^n\mathbb{P}\left(\faktor{L_i}{M_i}\right)
\end{tikzcd}
\end{equation}
We endow $\overline{\mathrm{Im}(\underline{f})}$ with the reduced scheme structure. We call this scheme the \textit{Mustafin degeneration associated to} $\Gamma$ and denote it by $\mathcal{M}(\Gamma)$.
\end{definition}

We note that when $M_i=\langle0\rangle$, this definition agrees with \cref{def:musta}.  Furthermore, we observe the following.

\begin{corollary}
Let $\Gamma=\left\{\left[\faktor{L_1}{M_1}\right],\dots,\left[\faktor{L_n}{M_n}\right]\right\}$ be a set of quotient lattices. Furthermore, let $V_i=\langle M_i\rangle_K$.
Then, the generic fibre $\mathcal{M}(\Gamma)_K$ is isomorphic to any generalised multi-view variety $M_{\underline{A}}$ with $\mathrm{ker}(A_i)=V_i$. Therefore, we call $\mathcal{M}(\Gamma)$ a Mustafin degeneration of the generalised multi-view varieties $M_{\underline{A}}$ with $\mathrm{ker}(A_i)=V_i$.
\end{corollary}

As a next step, we choose coordinates on $\mathcal{M}(\Gamma)$ similar to the procedure of choosing coordinates on Mustafin varieties discussed in \cref{subsec-must}, i.e. $\mathcal{M}(\Gamma)$ for $\Gamma=\left\{\left[\faktor{L_1}{M_1}\right],\dots,\left[\faktor{L_n}{M_n}\right]\right\}$ and $M_i=(0)_{\mathcal{O}_K}$. For this, we fix a reference lattice $L\in\mathfrak{B}_d^0$ and $\Gamma=\left\{\left[\faktor{L_1}{M_1}\right],\dots,\left[\faktor{L_n}{M_n}\right]\right\}$ a set of quotient lattices. Furthermore, we consider $g_i\in\mathrm{GL}(V)$, such that $g_iL=L_i$. This, induces a surjective homomorphism
\begin{equation}
g_i':L\to \faktor{L_i}{M_i}
\end{equation}
Denoting $N_i=\mathrm{ker}(g_i')$, the homomorphism $g_i'$ induces the isomorphism
\begin{equation}
\tilde{g}_i:\faktor{L}{N_i}\to\faktor{L_i}{V_i}
\end{equation}

We denote $W_i=N_i\otimes K$ and obtain the following commutative diagram
\begin{equation}
\begin{tikzcd}
\mathbb{P}(V) \arrow[dashed,"{\Delta\circ(g_1^{-1},\dots,g_n^{-1})}"]{rr} \arrow[dashed,"\underline{f}"]{d} & & \prod_{K}\mathbb{P}\left(\faktor{V}{W_i}\right) \arrow{d} \\
\prod_{\mathcal{O}_K}\mathbb{P}\left(\faktor{L_i}{M_i}\right) \arrow["{(\tilde{g}_1^{-1},\dots,\tilde{g}_n^{-1})}"]{rr} & & \prod_{\mathcal{O}_K}\mathbb{P}\left(\faktor{L}{N_i}\right)
\end{tikzcd}
\end{equation}
where the vertical map on the right indicates the open immersions $\mathbb{P}\left(\faktor{V}{W_i}\right)\to\mathbb{P}\left(\faktor{L}{N_i}\right)$ for each factor. Then, $\mathcal{M}(\Gamma)$ is isomorphic to $\overline{\mathrm{Im}(\Delta\circ(g_1^{-1},\dots,g_n^{-1}))}\subset\prod_{\mathcal{O}_K}\mathbb{P}\left(\faktor{L}{N_i}\right)\cong\prod_{i=1}^{n}\mathbb{P}_{\mathcal{O}_K}^{d-1-\mathrm{rank}(M_i)}$ endowed with the reduced scheme structure.

\begin{example}
\label{ex-1}
Let $K=\mathbb{C}((t))=\{\sum_{i=n}^\infty\alpha_it^i\mid n\in\mathbb{Z}\}$, then we have $\mathcal{O}_K=\mathbb{C}[[t]]=\{\sum_{i=n}^\infty\alpha_it^i\mid n\in\mathbb{Z}_{\ge0}\}$, $k=\mathbb{C}$ and let $e_1,e_2,e_3$ be the standard basis of $K^3$. We fix the following three lattices
\begin{align}
&L_1=\mathcal{O}_K e_1+\mathcal{O}_K e_2+\mathcal{O}_K e_3, \quad M_1=\mathcal{O}_K e_3,\quad L_2=\mathcal{O}_K e_1+\mathcal{O}_K e_2+\mathcal{O}_K t^{-1}e_3,\\
&M_2=\mathcal{O}_K e_2,\quad  L_3=\mathcal{O}_K e_1+\mathcal{O}_K t^{-1}e_2+\mathcal{O}_K e_3,\quad  M_3=\mathcal{O}_K e_1
\end{align}

and  
\begin{equation}
\Gamma=\left\{\left[\faktor{L_1}{M_1}\right],\left[\faktor{L_2}{M_2}\right],\left[\faktor{L_3}{M_3}\right]\right\}.
\end{equation}

Then, in the notation of the above choice of coordinates, we choose the reference lattice as $L=L_1$ and
\begin{equation}
    g_1=\mathrm{id},\quad g_2=\begin{pmatrix}
    1 & 0 & 0\\
    0 & 1 & 0\\
    0 & 0 & t^{-1}
    \end{pmatrix},\quad
    g_3=\begin{pmatrix}
    1 & 0 & 0\\
    0 & t^{-1} & 0\\
    0 & 0 & 1
    \end{pmatrix}.
\end{equation}

Thus, we have $N_1=\mathcal{O}_Ke_3$, $N_2=\mathcal{O}_Ke_2$ and $N_3=\mathcal{O}_Ke_1$. Then, again in the notation of the above choice of coordinates, we have that $\mathcal{M}(\Gamma)$ is isomorphic to the closure of the image of
\begin{equation}
\begin{tikzcd}
    f\colon\mathbb{P}^2_K \arrow[dashed]{r}& \mathbb{P}^1_{\mathcal{O}_K}\times\mathbb{P}^1_{\mathcal{O}_K}\times\mathbb{P}^1_{\mathcal{O}_K}\\
    (x_0\colon x_1\colon x_2) \arrow[mapsto]{r}&  \left(A_1\begin{pmatrix} x_0\\x_1\\x_2\end{pmatrix},A_2\begin{pmatrix} x_0\\x_1\\x_2\end{pmatrix},A_2\begin{pmatrix} x_0\\x_1\\x_2\end{pmatrix}\right)
\end{tikzcd}
\end{equation}

for
\begin{equation}
    A_1=\begin{pmatrix} 1 & 0 & 0\\ 0 & 1 & 0 \end{pmatrix},\quad A_2=\begin{pmatrix} 1 & 0 & 0\\ 0 & 0 & t \end{pmatrix},\quad A_3=\begin{pmatrix} 0 & t & 0\\ 0 & 0 & 1 \end{pmatrix}.
\end{equation}

Let $\left(x_{ij}\right)_{\substack{i=1,2\\j=1,2,3}}$ be the coordinates of $\mathbb{P}^1_{\mathcal{O}_K}\times\mathbb{P}^1_{\mathcal{O}_K}\times\mathbb{P}^1_{\mathcal{O}_K}$. Then, we have that 
\begin{equation}
    \mathcal{M}(\Gamma)\subset\mathbb{P}^1_{\mathcal{O}_K}\times\mathbb{P}^1_{\mathcal{O}_K}\times\mathbb{P}^1_{\mathcal{O}_K}
\end{equation}
 is cut out by the ideal
\begin{equation}
    \langle t^2x_{21}x_{12}x_{23}-x_{11}x_{22}x_{13}\rangle_{\mathcal{O}_K[x_{ij}]}.
\end{equation}
The special fibre $\mathcal{M}(\Gamma)_k$ is cut out by
\begin{equation}
    \langle x_{11}x_{22}x_{13}\rangle_{k[x_{ij}]}=\langle x_{11}\rangle_{k[x_{ij}]} \cap \langle x_{22}\rangle_{k[x_{ij}]} \cap \langle x_{13}\rangle_{k[x_{ij}]},
\end{equation}
i.e. $\mathcal{M}(\Gamma)_k$ is the union of three copies of $\mathbb{P}^1\times\mathbb{P}^1$.
\end{example}

Before deriving our first main result, we need the following definition.

\begin{definition}
Let $\Gamma=\left\{[\faktor{L_1}{M_1}],\dots,\faktor{L_n}{M_n}\right\}$ be a set of quotient lattices and let $V_i=\langle M_i\rangle_K\subset V$ be the subspace spanned by $M_i$. We then set
\begin{equation}
d_I=\mathrm{dim}\bigcap_{i\in I}V_i
\end{equation}
for all $I\subset \{1,\dots,n\}$. We further define for each $r\in \mathbb{Z}_{\ge0}$ the set
\begin{equation}
    N(r)\coloneqq\{\underline{m}\in\mathbb{Z}_{\ge0}\mid d-\sum_{i\in I} m_i>d_I\}.
\end{equation}
Let $r_0$ be the maximal $r$, such that $N(r)$ is non-zero. We then define $N_\Gamma\coloneqq N(r_0)$.
\end{definition}

The following is our first main result, where we collect the geometric properties of $\mathcal{M}(\Gamma)$.

\begin{theorem}
\label{thm-main1}
Let $\Gamma=\left\{\left[\faktor{L_1}{M_1}\right],\dots,\left[\faktor{L_n}{M_n}\right]\right\}$ be a set of quotient lattices. Then $\mathcal{M}(\Gamma)$ is an integral, flat and projective scheme over $\mathrm{Spec}(\mathcal{O}_K)$. Moreover, its special fibre $\mathcal{M}(\Gamma)_k$ is connected, equidimensional of dimension $\mathcal{M}(\Gamma)_K$ and generically reduced. Moreover, it has at most $|N_\Gamma|$ irreducible components.\\
Finally, if $\mathrm{rank}(L_i)=4$ and $\langle M_i\rangle_K$ are generic subspaces of $V$ of dimension $1$, then $\mathcal{M}(\Gamma)_k$ is reduced and Cohen-Macaulay.
\end{theorem}

\begin{proof}
The scheme $\mathcal{M}(\Gamma)$ is projective and reduced by construction. Since $\mathcal{O}_K$ is a discrete valuation ring, flatness follows as $\mathcal{M}(\Gamma)$ is reduced with non-empty generic fibre \cite[Proposition 4.3.9]{liubook}. As $\mathcal{M}(\Gamma)$ is flat, its special fibre is equidimensional of the same dimension as the generic fibre \cite[Proposition 4.416]{liubook}.\\
In order to prove that $\mathcal{M}(\Gamma)$ is connected we use Zariski's connectedness principle \cite[Theorem 5.3.15]{liubook}. In our situation, this means that once we to prove $\mathcal{O}_{\mathcal{M}(\Gamma)}(\mathcal{M}(\Gamma)_K)=K$ and $\mathcal{O}_{\mathcal{M}(\Gamma)}(\mathcal{M}(\Gamma))=\mathcal{O}_K$, the connectedness of $\mathcal{M}(\Gamma)_k$ follows. We see that $\mathcal{M}(\Gamma)_K$ is a projective variety, therefore $\mathcal{O}_{\mathcal{M}(\Gamma)}(\mathcal{M}(\Gamma)_K)=\mathcal{O}_{\mathcal{M}(\Gamma)_K}(\mathcal{M}(\Gamma)_K)=K$ follows immediately. Moreover, by \cite[Theorem 5.3.2 (a)]{liubook} we have that $\mathcal{O}_{\mathcal{M}(\Gamma)}(\mathcal{M}(\Gamma)_K)$ is a finitely generated $\mathcal{O}_K$ module and it is contained in $\mathcal{O}_{\mathcal{M}(\Gamma)}(\mathcal{M}(\Gamma)_K)=K$. Since $\mathcal{O}_K$ is integrally closed, we have $\mathcal{O}_{\mathcal{M}(\Gamma)}(\mathcal{M}(\Gamma))=\mathcal{O}_K$.\\
Let $\mathrm{rank}(L_i))=l_i$, then by \cref{thm-bing} we have that the Chow class of $\mathcal{M}(\Gamma)_K$ is
\begin{equation}
\label{equ-proofclass}
    \prod_{\underline{m}\in N_\Gamma} H_i^{l_i-1-m_i}.
\end{equation}
As $\mathcal{M}(\Gamma)_k$ is a specialisation of $\mathcal{M}(\Gamma)_K$, we have by \cite[Section 20.3]{fulton2013intersection} that the Chow class of $\mathcal{M}(\Gamma)_k$ agrees with that of $\mathcal{M}(\Gamma)_K$. Moreover, the Chow class of $\mathcal{M}(\Gamma)_k$ is the sum of the Chow classes of its components. As the Chow class of each component is effective, this is a sum non-negative monomials in the $H_i$. Therefore, the number of components is at most the number of terms in \cref{equ-proofclass}, which is equal to the cardinality of $N_\Gamma$. Moreover, since the coefficients of all monomials in \cref{equ-proofclass}, the special fibre $\mathcal{M}(\Gamma)_k$ is generically reduced.\\
Finally, let $\mathrm{rank}(L_i)=4$ and $\langle M_i\rangle_K$ generic subspaces of $V$ of dimension $1$. In this case, $\mathcal{M}(\Gamma)_K$ (resp. $\mathcal{M}(\Gamma)_k$) is a $K-$valued (resp. $k$-valued) point of the Hilbert scheme $\mathcal{H}_n$ studied in \cite{aholt2013hilbert}. By \cite[Corollary 6.1]{aholt2013hilbert}, the scheme corresponding to each such point is reduced and Cohen-Macaulay. Thus, the special fibre $\mathcal{M}(\Gamma)$ is reduced. As the fibres of $\mathcal{M}(\Gamma)$ are Cohen-Macaulay and the scheme is flat, it follows from \cite[Lemma 37.20.2]{SP} that $\mathcal{M}(\Gamma)$ is Cohen-Macaulay, as well.
\end{proof}

\section{Irreducible components of $\mathcal{M}(\Gamma)_k$}
In this section, we study the irreducible components of $\mathcal{M}(\Gamma)_k$ and give a partial classification of these components.

\begin{lemma}
Let $\Gamma\subset\Gamma'$ be sets of quotient lattices, such that $\mathrm{dim}(\mathcal{M}(\Gamma)_K)=\mathrm{dim}(\mathcal{M}(\Gamma')_K)$. Let $C$ be an irreducible component of $\mathcal{M}(\Gamma)_k$, then there is a unique irreducible component $C'$ of $\mathcal{M}(\Gamma')_k$, such that $C'$ maps birationally onto $C$ via the natural projection $\mathcal{M}(\Gamma')\to\mathcal{M}(\Gamma)$ induced by $\prod_{[L]\in\Gamma'}\mathbb{P}(L)\to\prod_{[L]\in\Gamma}\mathbb{P}(L)$.\\
In particular, for any $\left[\faktor{L_i}{M_i}\right]\in\Gamma$ with $\mathrm{rank}\left(\faktor{L_i}{M_i}\right)=\mathrm{dim}(\mathcal{M}(\Gamma)_K)+1=\mathrm{dim}(\mathcal{M}(\Gamma')_K)+1$, there exists a unique irreducible component $C_i\in\mathcal{M}(\Gamma)_k$, such that $C_i$ projects onto $\mathbb{P}\left(\faktor{L_i}{M_i}\right)_k$ via the natural projection.
\end{lemma}

\begin{proof}
Let $\Gamma=\left\{\left[\faktor{L_1}{M_1}\right],\dots,\left[\faktor{L_n}{M_n}\right]\right\}$, $\Gamma'=\left\{\left[\faktor{L_1}{M_1}\right],\dots,\left[\faktor{L_{n'}}{M_{n'}}\right]\right\}$ and $V_i=\left\langle M_i\right\rangle_{K}$. Recall, that the generic fibre of $\mathcal{M}(\Gamma')$ (resp. $\mathcal{M}(\Gamma))$ is given by the closure of the image of the rational map $\mathbb{P}(V)\dashrightarrow\prod_{i=1}^{n'}\mathbb{P}(\faktor{V}{V_i})$ (resp. $\mathbb{P}(V)\dashrightarrow\prod_{i=1}^{n}\mathbb{P}\left(\faktor{V}{V_i}\right)$). We see immediately that the map $\mathbb{P}(V)\dashrightarrow\prod_{i=1}^{n}\mathbb{P}\left(\faktor{V}{V_i}\right)$ factors via
\begin{equation}
    \begin{tikzcd}
    \mathbb{P}(V) \arrow[dashed]{r} \arrow[dashed]{dr}& \prod_{i=1}^{n'}\mathbb{P}\left(\faktor{V}{V_i}\right) \arrow["\eta"]{d}\\
    & \prod_{i=1}^{n}\mathbb{P}\left(\faktor{V}{V_i}\right)
    \end{tikzcd}
\end{equation}
where $\eta$ is the projection onto the first $n$ factors. Since the $\eta$ is a closed morphism it maps $\mathcal{M}(\Gamma')_K$ onto $\mathcal{M}(\Gamma)_K$. Moreover, it pushes the Chow class of $\mathcal{M}(\Gamma')_K$ forward to the Chow class of $\mathcal{M}(\Gamma)_K$. Therefore, under specialisation the Chow class of  $\mathcal{M}(\Gamma')_k$ pushes forward to the class of $\mathcal{M}(\Gamma)_k$.\\
We now observe that if $\underline{m}\in N(\Gamma)$, then $(\underline{m},0,\dots,0)\in N(\Gamma')$. Let $C$ be an irreducible component of $\mathcal{M}(\Gamma)_k$ contributing the monomial
\begin{equation}
\label{equ-monpr1}
    \prod_{i=1}^{n} H_i^{l_i-1-m_i}
\end{equation}
to the Chow class of $\mathcal{M}(\Gamma)_k$. Then, since $(\underline{m},0,\dots,0)\in N(\Gamma')$, there is a unique irreducible component $C'$ of $\mathcal{M}(\Gamma')_k$ which contributes the monomial
\begin{equation}
\label{equ-monpr2}
    \prod_{i=1}^{n}H_i^{l_i-1-m_i}\prod_{j=n+1}^{n'}H_i^{l_i-1}.
\end{equation}
We observe that \cref{equ-monpr2} pushes forward to \cref{equ-monpr1} under the projection $\mathcal{M}(\Gamma')\to\mathcal{M}(\Gamma)$. Since $C$ is the only component of $\mathcal{M}(\Gamma)_k$ contributing the monomial in \cref{equ-monpr1}, it must be the image of $C'$. As the coefficent of \cref{equ-monpr1} in the Chow class of $C$ is $1$, the projection restricted to $C'$ is birational. This finishes the proof of the first assertion.\\
The second assertion follows from the first assertion in the case $\Gamma=\left\{\left[\faktor{L_i}{M_i}\right]\right\}$.
\end{proof}

Motivated by the above lemma, we give the following definition.

\begin{definition}
\label{def-comp}
Let $\Gamma=\left\{\left[\faktor{L_1}{M_1}\right],\dots,\left[\faktor{L_n}{M_n}\right]\right\}$ be a set of quotient lattices and let $C$ be an irreducible component of $\mathcal{M}(\Gamma)_k$. We call $C$ \textit{primary}, if it projects birationally onto $\mathbb{P}\left(\faktor{L_i}{M_i}\right)_k$ for some $i$ under the map $\mathcal{M}(\Gamma)\to\mathbb{P}\left(\faktor{L_i}{M_i}\right)$.\\
We further call $C$ \textit{secondary}, if it is not a primary component of $\mathcal{M}(\Gamma)_k$ and if there exists a class of quotient lattice $\left[\faktor{L}{M}\right]$, such that there exists a primary component of $\mathcal{M}\left(\Gamma\cup\left\{\left[\faktor{L}{M}\right]\right\}\right)_k$ projecting birationally onto $\mathbb{P}\left(\faktor{L}{M}\right)$, also projects birationally onto $C$ under $\mathcal{M}\left(\Gamma\cup\left\{\left[\faktor{L}{M}\right]\right\}\right)\to\mathcal{M}(\Gamma)$.\\
We call all other irreducible components of $\mathcal{M}(\Gamma)_k$ \textit{tertiary}.
\end{definition}

\begin{remark}
We note that our terminology in \cref{def-comp} is inspired from the notions of primary, secondary and tertiary components in \cite{haebich}, where Mustafin degenerations of flag varieties were studied.
\end{remark}

Our next goal is the following proposition, where we classify the primary and secondary components of $\mathcal{M}(\Gamma)_k$.

\begin{proposition}
\label{prop-class}
Let $\Gamma$ be a finite set of quotient lattices. The primary and secondary components of $\mathcal{M}(\Gamma)_k$ are generalised multi-view varieties over $k$.
\end{proposition}

Our first step towards proving \cref{prop-class} is the following construction.

\begin{construction}
\label{con:var}
Let $\Gamma=\left\{\left[\faktor{L_1}{M_1}\right],\dots,\left[\faktor{L_n}{M_n}\right]\right\}$ be a finite set of quotient-lattices. Let $L$ be the reference lattice as in the choice of coordinates on $\mathcal{M}(\Gamma)$ above, $g_i\in\mathrm{GL}(V)$ with $g_iL=L_i$ and $N_i=g_i^{-1}(M_i)$. Then, the rational map
\begin{equation}
\begin{tikzcd}
    \mathbb{P}(V) \arrow[dashed,"{(\tilde{g}_1,\dots,\tilde{g}_n)}\circ \Delta"]{rrr}& & & \mathbb{P}\left(\faktor{L}{N_i}\right)
\end{tikzcd}
\end{equation}
extends uniquely to a rational map
\begin{equation}
\begin{tikzcd}
    \mathbb{P}(L) \arrow[dashed]{rrr} & & & \mathbb{P}\left(\faktor{L}{N_i}\right).
\end{tikzcd}
\end{equation}
In coordinates this extension may be described as follows. Let $G_i$ be a matrix representation of the module homomorphism
\begin{equation}
    \tilde{g}_i\colon L\to\faktor{L}{N_i}.
\end{equation}
We define 
\begin{equation}
s_i=\mathrm{min}_{\mu,\nu=1,\dots,d}(\mathrm{val}((G_i)_{\mu\nu}))
\end{equation}
 and $G_i'=\pi^{-s_i}G_i$. Then $G_i'$ is defined over $\mathcal{O}_K$, but $\pi^{-1}G_i'$ is not. We first note that $G_i'$ and $G_i$ induce the same morphism
 \begin{equation}
 \begin{tikzcd}
  \tilde{G}_i\colon\mathbb{P}(L) \arrow[dashed]{r} & \mathbb{P}\left(\faktor{L}{N_i}\right)
 \end{tikzcd}
 \end{equation}
 since they lie in the same equivalence class of $\mathrm{PGL}(V)$. Moreover, for the map over the generic fibre we have that $\tilde{G}_{i|\mathbb{P}(L)_K}\equiv\tilde{g}_i$, while the map of the map $\tilde{G}_{i|\mathbb{P}(L)_k}$ over the special fibre is given by the matrix $H_i=G_i'\,\mathrm{mod}\, \pi$ over $k$.\\
 We then consider the map over the residue field
 \begin{equation}
 \begin{tikzcd}
      (H_1,\dots,H_n)\circ\Delta\colon\mathbb{P}(L)_k \arrow[dashed,"{(H_1,\dots,H_n)\circ\Delta}"]{rrr} & & & \prod_{i=1}^n\mathbb{P}\left(\faktor{L}{N_i}\right)_k
 \end{tikzcd}
 \end{equation}
and denote $X_{[L],\Gamma,\underline{g}}=\overline{\mathrm{Im}((H_1,\dots,H_n)\circ\Delta)}$.
\end{construction}

\begin{lemma}
\label{lem-comp1}
Let $\Gamma=\left\{\left(\faktor{L_1}{M_1}\right],\dots,\left[\faktor{L_n}{M_n}\right]\right\}$, $L$ a reference lattice and $g_i\in\mathrm{GL}(V)$ with $g_iL=L_i$. Then, we have
\begin{equation}
    X_{[L],\Gamma,\underline{g}}\cong X_{[L],\Gamma,\underline{h}}.
\end{equation}
for any $\underline{h}=(h_1,\dots,h_n)\in\mathrm{GL}(V)^n$ with $h_iL=L_i$. Therefore, we denote $X_{[L],\Gamma}\coloneqq X_{[L],\Gamma,\underline{g}}$.\\
Furthermore, we consider $\mathcal{M}(\Gamma)$ embedded into $\prod_{i=1}^n\mathbb{P}\left(\faktor{L}{N_i}\right)$ via $(\tilde{g}_1,\dots,\tilde{g}_n)\circ\Delta$. Then, we have
\begin{equation}
X_{[L],\underline{g}}\subset\mathcal{M}(\Gamma)_k\subset\prod_{i=1}^n\mathbb{P}\left(\faktor{L}{N_i}\right)_k
\end{equation}
and $X_{[L],\Gamma}$ is an irreducible component of $\mathcal{M}(\Gamma)$ if and only if $\mathrm{dim}(X_{[L],\Gamma})=\mathrm{dim}(\mathcal{M}(\Gamma)_K)=\mathrm{dim}(\mathcal{M}(\Gamma)_k)$.
\end{lemma}

\begin{proof}
The proof is analogous to that of \cite[Lemma 3.3, Lemma 3.7]{hahn2017mustafin}.
\end{proof}

Before proving \cref{prop-class}, we illustrate \cref{lem-comp1} in the following example.

\begin{example}
We consider the same notation as in \cref{ex-1}. Now, we show that all irreducible components of $\mathcal{M}(\Gamma)_k$ in \cref{ex-1} are varieties of the form $X_{[L],\Gamma}$ for some reference lattice $L$.  First, we fix a new reference lattice $L=\mathcal{O}_Kt^{-1}e_1+\mathcal{O}_Kt^{-1}e_2+\mathcal{O}_Ke_3$. Then, we may choose $g_i\in\mathrm{GL}(V)$ with $g_iL=L_i$ as follows
\begin{equation}
    g_1=\begin{pmatrix} t & 0 & 0\\0 & t & 0\\ 0 & 0 & 1\end{pmatrix},\quad g_2=\begin{pmatrix} t & 0 & 0\\0 & t & 0\\ 0 & 0 & t^{-1}\end{pmatrix},\quad g_3=\begin{pmatrix} t & 0 & 0\\0 & 1 & 0\\ 0 & 0 & 1\end{pmatrix}.
\end{equation}
In the notation of \cref{con:var}, we obtain
\begin{equation}
    \tilde{G}_1=\begin{pmatrix} 1 & 0 & 0\\0 & 1 & 0\end{pmatrix},\quad 
    \tilde{G}_2=\begin{pmatrix} 1 & 0 & 0\\ 0 & 0 & t^{2}\end{pmatrix},\quad 
    \tilde{G}_3=\begin{pmatrix} 0 & 1 & 0\\ 0 & 0 & 1\end{pmatrix}.
\end{equation}
and
\begin{equation}
    H_1=\begin{pmatrix} 1 & 0 & 0\\0 & 1 & 0\end{pmatrix},\quad 
    H_2=\begin{pmatrix} 1 & 0 & 0\\ 0 & 0 & 0\end{pmatrix},\quad 
    H_3=\begin{pmatrix} 0 & 1 & 0\\ 0 & 0 & 1\end{pmatrix}.
\end{equation}
Thus, we obtain $X_{[L],\Gamma}=X_{[L],\Gamma,\underline{g}}$ is the irreducible component of $\mathcal{M}(\Gamma)_k$ cut out by $\langle x_{22}\rangle_{k[x_{ij}]}$. Similarly, we obtain the irreducible component cut out by $\langle x_{11}\rangle_{k[x_{ij}]}$ for $L=\mathcal{O}_Kt^{-1}e_1+\mathcal{O}_Ke_2+\mathcal{O}_Kt^{-1}e_3$ and the irreducible component cut out by $\langle x_{13}\rangle_{k[x_{ij}]}$ for $L=\mathcal{O}_Ke_1+\mathcal{O}_Kt^{-1}e_2+\mathcal{O}_Kt^{-1}e_3$.
\end{example}

\begin{proof}[Proof of \cref{prop-class}]
We first treat the case of primary components. To begin with, let $\Gamma=\left\{\left[\faktor{L_1}{M_1}\right],\dots,\left[\faktor{L_n}{M_n}\right]\right\}$ and let $L=L_i$. Then, in the setting of \cref{con:var} we may choose $g_i=\mathrm{id}$. We then consider the following commutative diagram
\begin{equation}
    \begin{tikzcd}
    \mathbb{P}(L)_k \arrow[dashed,swap,"H_i"]{drr} \arrow[dashed,"{(H_1,\dots,H_n)\circ\Delta}"]{rr}& & \prod_{i=1}^n\mathbb{P}\left(\faktor{L}{N_i}\right)_k \arrow["p_i"]{d}\\
    & & \mathbb{P}\left(\faktor{L}{N_i}\right)_k
    \end{tikzcd}
\end{equation}
where the vertical map $p_i$ is the projection to the $i-$th factor. The key observation is that $H_i$ is a surjective map. Therefore, since $H_i$ and $p_i\circ(H_1,\dots,H_n)\circ\Delta$ coincide on a dense open subset of $\mathbb{P}(L)_k$, we have that $X_{[L],\Gamma}=\overline{\mathrm{Im}((H_1,\dots,H_n)\circ\Delta)}$ projects onto $\mathbb{P}\left(\faktor{L}{N_i}\right)$. Therefore, if $\mathrm{rank}\left(\faktor{L}{N_i}\right)-1=\mathrm{dim}(\mathcal{M}(\Gamma)_k)$, we have that $X_{[L],\Gamma}$ is the unique $i-$th primary component. As $X_{[L],\Gamma}$ is a generalised multi-view variety, this settles the proposition for primary components.\\
Now, let $C$ be a secondary component of of $\mathcal{M}(\Gamma)_k$. Then, there exists an equivalence class of quotient lattices $\faktor{L_{n+1}}{M_{n+1}}$, such that for $\Gamma'=\Gamma\cup\left\{\left[\faktor{L_{n+1}}{M_{n+1}}\right]\right\}$ we have that the $n+1-$st primary components of $\mathcal{M}(\Gamma')_k$ projects onto $C$ under the projection $\prod_{i=1}^{n+1}\mathbb{P}\left(\faktor{L_i}{N_i}\right)\to\prod_{i=1}^{n+1}\mathbb{P}\left(\faktor{L_i}{N_i}\right)$. We now consider the following commutative diagram
\begin{equation}
     \begin{tikzcd}
    \mathbb{P}(L)_k \arrow[dashed,swap,"{(H_1,\dots,H_{n})\circ\Delta}"]{drr} \arrow[dashed,"{(H_1,\dots,H_{n+1})\circ\Delta}"]{rr}& & \prod_{i=1}^{n+1}\mathbb{P}\left(\faktor{L}{N_i}\right)_k \arrow["p_{1,\dots,n}"]{d}\\
    & & \prod_{i=1}^{n}\mathbb{P}\left(\faktor{L}{N_i}\right)_k
    \end{tikzcd}
\end{equation}
where the vertical map is the projection onto the first $n$ factors. By assumption, we know that $X_{[L_{n+1}],\Gamma'}=\overline{\mathrm{Im}((H_1,\dots,H_{n+1})\circ\Delta)}$ projects onto $C$ via $p_{1,\dots,n}$. Therefore, we obtain
\begin{equation}
C=\overline{\mathrm{Im}(p_{1,\dots,n}\circ(H_1,\dots,H_{n+1})\circ\Delta)}=\overline{\mathrm{Im}((H_1,\dots,H_{n})\circ\Delta)}.
\end{equation}
Thus, we see that $C$ is a generalised multiview variety, which finishes the proof.
\end{proof}

\begin{remark}
By the arguments in the proof of \cref{prop-class}, the irreducible components in \cref{ex-1} are secondary components.
\end{remark}

We end this section with the following remark.

\begin{remark}
In \cite{cartwright2011mustafin}, it was proved that all irreducible components are secondary or primary components when $M_i=\langle0\rangle$ for all $i$. By the results in \cite{hahn2017mustafin}, all varieties of the form $X_{[L],\Gamma}$ of the correct dimension are irreducible components of $\mathcal{M}(\Gamma)_k$ for some set of lattices $\Gamma$. It would be interesting to know, whether such a classification extends to the case of $\Gamma$  a set of quotient lattices.
\end{remark}

We finish with the following example.
\begin{example}
\label{ex-genericin}
The main result of \cite{aholt2013hilbert} is that multi-view varieties are dense in an irreducible component of a certain Hilbert scheme $\mathcal{H}_n$. The special fibres of Mustafin degenerations of multi-view varieties will be $k-$valued points of this Hilbert scheme. It would be interesting to know which $k-$valued points are reached. Evidence that all $k-$valued points are reached is given in \cite[corollary 3.9]{cartwright2011mustafin} where -- in a different setting -- this was shown for $\mathrm{rank}(L_i)=3$ and $M_i=\langle0\rangle$. Furthermore, for $\mathrm{rank}(L_i)=4$ and $M_i=\langle0\rangle$ it was shown in \cite[proposition 6.8]{cartwright2011mustafin} that all monomial ideals in $\mathcal{H}_n$ are reached via Mustafin degenerations.\\
In the following, we give further evidence for this claim to hold, by showing that the \textit{generic initial ideal} $M_n$, which lies in every irreducible component of $\mathcal{H}_n$ (see \cite[section 3]{cartwright2010hilbert} and \cite[corollary 6.1]{cartwright2010hilbert}) is attained as a special fibre of a Mustafin degeneation.  The generic initial ideal $M_n\subset\mathbb{C}\left[\left(x_{ij}\right)_{\substack{i=1,2,3\\j=1,\dots,n}}\right]$ is generated by the $\binom{n}{2}$ quadrics $x_{3i}x_{3j}$, the $\binom{n}{3}$ cubics $x_{3i}x_{2j}x_{2l}$ and $\binom{n}{4}$ quartics $x_{2i}x_{2j}x_{2l}x_{2m}$ for pairwise distinct indices $i,j,l,m\in\{1,\dots,n\}$.\\
We fix $K=\mathbb{C}((t))$. As noted above, this yields $\mathcal{O}_K=\mathbb{C}[[t]]$, $\mathfrak{m}_K=\{\sum_{l\ge n}a_lt^l\mid a_l\in\mathbb{C},n\in\mathbb{Z}_{>0}\}$ and $K\cong\mathbb{C}$.
Furthermore, we fix $L=\mathcal{O}_K^4=\mathcal{O}_Ke_1+\dots+\mathcal{O}_Ke_4$ as a reference lattice, where $e_1,\dots,e_4$ is the standard basis. Let $\left(a_{ij}^{(l)}\right)_{\substack{i,j=1,\dots,4\\l=1,\dots,4}}\in\mathbb{Q}$ and consider the following three matrices
\begin{align}
    &A_r=\begin{pmatrix}
    1 \\
    & 1 \\
    & & t\\
    & & & t^2
    \end{pmatrix}
    \begin{pmatrix}
    a_{11}^{(r)} & \hdots & a_{14}^{(r)}\\ 
    \vdots & \ddots & \vdots\\ 
    a_{41}^{(r)} & \hdots & a_{44}^{(r)}
    \end{pmatrix}=
    \begin{pmatrix}
    a_{11}^{(r)} & \hdots & a_{14}^{(r)}\\ 
    a_{21}^{(r)} & \hdots & a_{24}^{(r)}\\ 
    ta_{31}^{(r)} & \hdots & ta_{34}^{(r)}\\ 
    t^2a_{41}^{(r)} & \hdots & t^2a_{44}^{(r)}
    \end{pmatrix}
    \end{align}
for $r=1,\dots,n$. For generic $(a_{ij}^{(l)})_{i,j,l}$, the matrices $A_1,\dots,A_n$ are invertible, i.e. there is a Zarisiki open subset $U\subset \mathbb{C}^{12n}$, such that for all $(a_{ij}^{(l)})_{i,j,l}\in U$, the matrices $A_1,\dots,A_n$ are invertible. We assume that $(a_{ij}^{(l)})_{i,j,l}\in U$ and define $g_i=A_i^{-1}$. We define lattices $L_i=g_iL$. Furthermore, we denote $N_i=\langle e_i\rangle$ and definte $M_i=g_i\cdot e_i$.\\
This yields the set of quotient lattices $\Gamma=\left\{\left[\faktor{L_1}{M_1}\right],\dots,\left[\faktor{L_n}{M_n}\right]\right\}$. As observed in \cref{sec-con}, the Mustafin degeneration $\mathcal{M}(\Gamma)$ is isomorphic to the closure of the following rational map endowed with the reduced scheme structure.
\begin{equation}
\begin{tikzcd}
    \underline{g}^{-1}\circ\Delta\colon\mathbb{P}_K^3 \arrow[dashed,"{(g_1^{-1},\dots,g_n^{-1})\circ\Delta}"]{rrr} & & & \mathbb{P}\left(\faktor{L}{N_1}\right)\times\dots\times\mathbb{P}\left(\faktor{L}{N_n}\right)\cong\mathbb{P}(\tilde{L})^n,
\end{tikzcd}
\end{equation}
where $\tilde{L}=\mathcal{O}_Ke_2+\mathcal{O}_Ke_3+\mathcal{O}_Ke_4$.\\
This may be computed as follows. We define
\begin{align}
&B_r=\begin{pmatrix}
1 \\
& t\\
& & t^2
\end{pmatrix}
\begin{pmatrix}
a_{21}^{(r)} & \hdots & a_{24}^{(r)}\\ 
a_{31}^{(r)} & \ddots & a_{34}^{(r)}\\ 
a_{41}^{(r)} & \hdots & a_{44}^{(r)}
\end{pmatrix}=
\begin{pmatrix}
a_{21}^{(r)} & \hdots & a_{24}^{(r)}\\ 
ta_{31}^{(r)} & \hdots & ta_{34}^{(r)}\\ 
t^2a_{41}^{(r)} & \hdots & t^2a_{44}^{(r)}
\end{pmatrix}
\end{align}
for $r=1,\dots,4$. Now, we consider the map
\begin{equation}
\begin{tikzcd}
    \underline{B}\circ\Delta\colon\mathbb{P}_K^3 \arrow[dashed,"{(B_1,\dots,B_4)\circ\Delta}"]{rrr}& & & \left(\mathbb{P}_K^2\right)^4
\end{tikzcd}
\end{equation}
and denote the coordinates of the $l-$th factor of $\left(\mathbb{P}_K^2\right)^3$ by $x_{12},x_{2l},x_{3l}$. Let $J$ the ideal in $K[x_{ij}]$ cutting out the variety $\overline{\mathrm{Im}(\underline{B}\circ\Delta)}$. We denote $J'=J\cap\mathcal{O}_K[x_{ij}]$ and obtain that $\mathcal{M}(\Gamma)$ is isomorphic the subscheme of $\mathbb{P}(\tilde{L})^n\cong\left(\mathbb{P}_{\mathcal{O}_K}^2\right)^n$ cut out by $J'$. Moreover, the special fibre $\mathcal{M}(\Gamma)_k$ is isomorphic to the subscheme of $\left(\mathbb{P}^2_{k}\right)^n$ cut out by $\tilde{J}=J'_{|t=0}\subset \mathbb{C}[x_{ij}]$.\\
Our next goal is to determine the ideal $\tilde{J}$. For this, we let $I=\{i_1,\dots,i_m\}\subset\{1,\dots,n\}$ and define the matrix
\begin{equation}
B_I=\begin{pmatrix}
B_{i_1} & p_{i_1} & 0 \hdots & & 0\\
B_{i_2} & 0 & p_{i_2} & \ddots & 0\\
\vdots & \vdots & \ddots & \ddots & \vdots\\
B_{i_m} & 0 & \hdots 0 & p_{i_m}\\
\end{pmatrix}
\end{equation}
where $p_{i_j}=(x_{1i_j},x_{2i_j},x_{3i_j})^T$. Then, by \cite[lemma 2.2]{aholt2013hilbert}, the ideal $J$ and therefore the ideal J' contains the maximal minors of the matrices $B_I$ for all $|I|\ge 2$.\\
A straightforward computation shows that the saturation with respect to $t$ of the maximal minor of $B_I$ for $I=\{i,j\}$ is of the form
\begin{equation}
    I_{ij}=\star\cdot x_{3i}x_{3j}+t\cdot h_{ij}(t,a_{ij}^{(l)},x_{ij})
\end{equation}
where $\star$ is a polynomial expression in the $(a_{ij}^{(l)})$ and $h_{ij}(t,a_{ij}^{(l)},x_{ij})$ is a polynomial expression in $t$, the $(a_{ij}^{(l)})$ and the $(x_{ij})$.\\
For $I=\{i,j,l\}$, the saturation with respect to $t$ of the maximal minor of $B_I$ obtained by deleting the $6-$th and $9-$th row is given by 
\begin{equation}
    I_{ijl}=\ast\cdot x_{3i}x_{2j}x_{2l}+t\cdot h_{ijl}(t,a_{ij}^{(l)},x_{ij})
\end{equation}
where $\ast$ is a polynomial expression in the $(a_{ij}^{(l)})$ and $h_{ijl}(t,a_{ij}^{(l)},x_{ij})$ is a polynomial expression in $t$, the $(a_{ij}^{(l)})$ and the $(x_{ij})$.\\
Furthermore, for $I=\{i,j,l,m\}$ the saturation with respect to $t$ of the maximal minor of $B_I$ obtained by deleting the $3$rd, $6-$th, $9-$th and $12-$th row is given by
\begin{equation}
    I_{ijlm}=\dagger\cdot x_{2i}x_{2j}x_{2l}x_{2m}+t\cdot h_{ijlm}(t,a_{ij}^{(l)},x_{ij})
\end{equation}
where $\dagger$ is a polynomial expression in the $(a_{ij}^{(l)})$ and $h_{ijlm}(t,a_{ij}^{(l)},x_{ij})$ is a polynomial expression in $t$, the $(a_{ij}^{(l)})$ and the $(x_{ij})$.\\
After possibly shrinking $U$, we obtain that the polynomial expressions $\star,\ast,\dagger$ do not vanish for all $a_{ij}^{(l)}$ in the Zariski open subset $U\subset\mathbb{C}^{12n}$. We see that the the $\binom{n}{2}$ quadrics $x_{3i}x_{3j}$ are obtained by $(I_{ij})_{|t=0}$, the $\binom{n}{3}$ cubics $x_{3i}x_{2j}x_{2l}$ by $(I_{ijl})_{|t=0}$ and the $\binom{n}{4}$ quartics $x_{2i}x_{2j}x_{2l}x_{2m}$ by $(I_{ijlm})_{|t=0}$. Therefore, we see that $M_n\subset \tilde{J}$. We note that $M_n$ is a radical ideal. Furthermore, by \cref{thm-main1} the ideal $\tilde{J}$ is radical, as well. Since both $M_n$ and $\tilde{J}$ are radical ideals and since by the discussion at the beginning of this example $M_n$ and $\tilde{J}$ share the same Hilbert function, it follows from $M_n\subset \tilde{J}$ that $M_n=\tilde{J}$.\\
Therefore, the generic initial ideal $M_n$ contained in all irreducible components of the Hilbert scheme $\mathcal{H}_n$ may be obtained as a special fibre of a Mustafin degeneration.
\end{example}

\printbibliography
\end{document}